\documentclass[twoside,11pt]{article}

\usepackage{graphicx, subfig}
\usepackage[top=1in,bottom=1in,left=1in,right=1in]{geometry}
\usepackage[sort&compress]{natbib} 
\usepackage{amsfonts, amsmath, amssymb, amsthm, constants, bbm,mathtools}
\usepackage{MnSymbol,float}
\usepackage{enumitem}
\usepackage[mathscr]{eucal}
\usepackage[linesnumbered,ruled,vlined]{algorithm2e}
\usepackage{booktabs}
\usepackage{graphicx}

\usepackage{color}
\definecolor{darkred}{RGB}{100,0,0}
\definecolor{darkgreen}{RGB}{0,100,0}
\definecolor{darkblue}{RGB}{0,0,150}

\DeclarePairedDelimiter\floor{\lfloor}{\rfloor}
\DeclarePairedDelimiter\ceil{\lceil}{\rceil}
\usepackage{hyperref}
\hypersetup{colorlinks=true, linkcolor=darkred, citecolor=darkgreen, urlcolor=darkblue}
\usepackage{url}

\newtheorem{thm}{Theorem}

\theoremstyle{remark}

\newtheorem{rem}{Remark}

\theoremstyle{definition}
\newtheorem{con}{Contribution}

\def\beq{\begin{equation}} 
\def\eeq{\end{equation}}
\def\beqn{\begin{eqnarray*}}
\def\eeqn{\end{eqnarray*}}
\def\Bitem{\begin{itemize}\setlength{\itemsep}{.2in}}
\def\bitem{\begin{itemize}\setlength{\itemsep}{.05in}}
\def\eitem{\end{itemize}}
\def\Benum{\begin{enumerate}\setlength{\itemsep}{.2in}}
\def\benum{\begin{enumerate}\setlength{\itemsep}{.05in}}
\def\eenum{\end{enumerate}}
\def\bmult{\begin{multline*}}
\def\emult{\end{multline*}}
\def\bcenter{\begin{center}}
\def\ecenter{\end{center}}
\def\bframe{\begin{frame}}
\def\eframe{\end{frame}}

\newcommand{\thmref}[1]{Theorem~\ref{thm:#1}}

\newcommand{\secref}[1]{Section~\ref{sec:#1}}
\newcommand{\figref}[1]{Figure~\ref{fig:#1}}

\newcommand{\algref}[1]{Algorithm~\ref{alg:#1}}

\DeclareMathOperator*{\argmax}{arg\, max}

\def\cS{\mathcal{S}}

\def\bX{\mathbf{X}}

\newcommand{\E}{\operatorname{\mathbb{E}}}
\renewcommand{\P}{\operatorname{\mathbb{P}}}

\def\eps{\varepsilon}

\def\symd{\triangle}

\def\1{\mathbbm{1}}

\def\scan{\textsc{scan}}
\def\mscan{\textsc{mscan}}

\definecolor{purple}{rgb}{0.4,.1,.9}

\urldef\eacurl\url{http://www.math.ucsd.edu/~eariasca/}
\urldef\ylurl\url{https://nozoeli.github.io/}

\begin{document}
	\thispagestyle{empty}
	
	\title{A Multiscale Scan Statistic for Adaptive Submatrix Localization}
	\author{
		Yuchao Liu\thanks{Microsoft Corporation --- \href{https://nozoeli.github.io/}{Homepage} }
		\and 
		Ery Arias-Castro\thanks{University of California, San Diego --- \href{http://www.math.ucsd.edu/\~eariasca/}{Homepage}}
	}
	\date{}
	\maketitle

\begin{abstract}
  We consider the problem of localizing a submatrix with larger-than-usual entry values inside a data matrix, without the prior knowledge of the submatrix size. We establish an optimization framework based on a multiscale scan statistic, and develop algorithms in order to approach the optimizer. We also show that our estimator only requires a signal strength of the same order as the minimax estimator with oracle knowledge of the submatrix size, to exactly recover the anomaly with high probability. We perform some simulations that show that our estimator has superior performance compared to other estimators which do not require prior submatrix knowledge, while being comparatively faster to compute.
\end{abstract}

\section{Introduction}

Observing a data matrix $\bX$, the problem of submatrix localization, a.k.a., biclustering or coclustering, consists in  localizing one or several submatrices whose entries are `unusually large', or significant in some other prescribed way. Such submatrices may be of special interest, since the unusual entries may indicate some potential association or relationship between the corresponding variables.  One important application is in the analysis of gene expression data \citep{cheng2000biclustering, tanay2002discovering}, where the row and column indices stand for the genes and conditions, respectively. See \citep{madeira2004biclustering,charrad2011simultaneous} for a survey and \citep{prelic2006systematic} for a comparison of existing methods.

\subsection{Submatrix localization} \label{sec:introlocal}

We denote by $M$ and $N$ the number of rows and columns of the data matrix $\bX = (X_{ij})$. We assume for simplicity that there is only one submatrix, of size $m^* \times n^*$, to be localized. We further assume that the entries are independent and normally distributed, and with same (known) variance, as well as homogeneous both inside and outside the anomalous submatrix, or more formally,
\beq \label{normal}
X_{ij} = \theta 1_{\{(i,j) \in( I^* \times J^*)\}} + \eps_{ij},
\eeq
where the $\eps_{ij}$ are IID standard normal distribution, while $I^* \subset [M] , J^* \subset [N]$ are the index sets defining anomalous submatrix.  Note that $|I^*| = m^*$ and $|J^*| = n^*$. 
Here we use the symbol $[M]$ to represent integer set $\{ 1,\ldots, M\}$, while $|I|$ is the cardinality of discrete set $I$. 
The parameter $\theta$, which is assumed to be strictly positive in this case, quantifies the per-element signal strength within the submatrix. 

This parametric setup is used, for example, in \citep{butucea2013detection, kolar2011minimax, butucea2015sharp}.  \citep{butucea2013detection} focuses on the simpler problem of detecting the presence of an anomalous submatrix and considers the case where the submatrix size is unknown, while \citep{kolar2011minimax, butucea2015sharp} focus on localizing the submatrix with full knowledge of its size $(m^*,n^*)$.  In all these papers, submatrices are `scanned' for anomaly.
In detail, the scan statistic, with $(m^*,n^*)$ \emph{known}, is defined as 
\beq\notag
\scan_{m^*,n^*} (\bX) = \max_{I\subset [M],|I| = m^*, J \subset [N] , |J| = n^*} \sum_{(i,j) \in  I \times J} X_{ij}.
\eeq 
Localizing the submatrix by the scan statistic is simply returning the row and column index sets achieving the maximum, namely:
\beq\label{scan}
\Phi_{\scan}(\bX) = \argmax_{I\subset [M],|I| = m^*, J \subset [N] , |J| = n^*} \sum_{(i,j) \in  I \times J} X_{ij}.
\eeq
This estimator has the `minimax' property under the above assumptions.  
\begin{thm}[\cite{butucea2015sharp} Theorems 2.1, 2.2] \label{thm:bis} 
Consider the model \eqref{normal} as described above in an asymptotic regime where 
\beq\label{assume}
M,N,m^*,n^* \to \infty, \quad \frac{\max(m^*, n^*)}{\min (M,N)} \to 0.
\eeq
Denote
\begin{multline*}
\theta_0 
= \max \bigg\{ \frac{\sqrt{2\log n^*} + \sqrt{2 \log (N-n^*)}}{\sqrt{m^*}}, \\
\frac{\sqrt{2\log m^*} + \sqrt{2 \log (M-m^*)}}{\sqrt{n^*}}, \\ \frac{\sqrt{2n^*\log(N/n^*) + 2m^*\log (M/m^*)}}{\sqrt{m^*n^*}}  \bigg\}.
\end{multline*}
Then if 
\beq\notag
\liminf \theta/\theta_0 > 1,
\eeq
the estimator \eqref{scan} is strongly consistent in the sense that 
\beq\notag
\P (\Phi_{\scan}(\bX) \neq I^* \times J^*) \to 0.
\eeq
Moreover, if 
\beq\label{lower}
\limsup \theta/\theta_0 < 1,
\eeq
there does not exist a strongly consistent estimator. 
\end{thm}

Condition \eqref{lower} provides a lower bound of the signal strength for the existence of a strongly consistent estimator with \emph{prior knowledge} of the anomaly size $(m^*, n^*)$. 

We consider here the case when the submatrix size $(m^*, n^*)$ is \emph{not known}.  Our strategy is straightforward: we consider a variant of the scan statistic, different from the one defined in \eqref{scan}, which scans submatrices of different sizes.

\begin{con}[Multiscale scan statistic and its property] 
We propose a multiscale scan statistic as an estimator under the parametric model defined in \secref{introlocal}. We prove that under some regularity conditions, this statistic exactly recovers the submatrix with high probability when the signal strength is of the same order of magnitude as that implied by \thmref{bis} under the setting where the submatrix size is known a priori. 
\end{con}

Scanning, even when the submatrix size is known, is computationally difficult.  It is in fact known to be NP-hard \citep[Thm 1]{cheng2000biclustering}. 
To tackle the problem of computing the scan statistic defined in \eqref{scan}, \citep{shabalin2009finding} develops an alternate maximization algorithm called LAS (for Large Average Submatrix).  Based on the LAS algorithm, we develop two algorithms in order to approach our proposed statistic.

\begin{con}
We develop iterative algorithms to solve the problem of finding the proposed multiscale scan statistic. One algorithm adapts the idea of \citep{shabalin2009finding} and is a hill-climbing type optimization algorithm. The other is a variant based on the golden section search. 
\end{con}

\subsection{More related work}

There is an active research line focusing on different aspects of submatrix localization. \citep{arias2017distribution, butucea2013detection} concentrate on the detection of the existence of such an anomalous submatrix. 
Besides \citep{butucea2015sharp, kolar2011minimax}, discussed above, \citep{hajek2017information} also considers the minimal signal strength needed for accurate localization under symmetric data structure, while \citep{hajek2017information,hajek2017submatrix} develop upper and lower bounds for the existence of weakly consistent estimators.  
A convex optimization framework for biclustering is proposed and associated algorithms are developed in \citep{chi2016convex}. 
A selective inference framework for quantifying the information contained within a selected submatrix is proposed by \citep{lee2015evaluating}. 

The computation issue is attracting an increasing amount of attention from researchers in the field, addressing the tradeoff between statistical power and computational tractability \citep{ma2015computational, chen2016statistical, cai2017computational}. 
In this context, a variety of computationally tractable (i.e., running in polynomial time) methods have been proposed and analyzed, such as methods relying on a semidefinite relaxation \citep{chen2016statistical}, spectral methods \citep{cai2017computational}, methods relying on message passing heuristics \citep{hajek2017submatrix}, and more \citep{brault2016fast, tan2014sparse}. 

Another closely related research area is network analysis, for example, stochastic blockmodels \citep{holland1983stochastic}. Similar to submatrix localization, there are works that study the problem of detection \citep{arias2014community,verzelen2015community,zhang2016minimax}, some that study the task of partitioning the graph \citep{mossel2016consistency}, for which spectral and SDP methods have been proposed \citep{chen2016statistical, abbe2016exact, chaudhuri2012spectral, mcsherry2001spectral}.

\subsection{Content}
The remaining of this paper is organized as follows. 
In \secref{mle} we introduce the multiscale scan statistic. 
\secref{algo} introduces the main algorithms for approximating  the proposed statistic: a hill-climbing search algorithm and a golden section search algorithm. Some theory is established The theoretical properties of the (exact) scan statistic are stated in \secref{theory}.
\secref{numerics} presents the result of some numerical simulations.

\section{The multiscale scan statistic}
\label{sec:mle}


The scan statistic as defined in \eqref{scan} requires knowledge of the size of the submatrix, which we have denoted $(m^*,n^*)$.  When this is unknown, we use a criterium for comparing sums over submatrices of different sizes.  For example, if two potential submatrix sizes $(m_1, n_1)$ and $(m_2, n_2)$ are investigated, we need to determine which one of $\scan_{m_1,n_1} (\bX)$ and $\scan_{m_2,n_2} (\bX)$ is more `significant'.  



When comparing two normalized scan statistics, especially when the sizes of the scan statistics differ significantly, we need to consider the effect of background noise into the comparison. 
Inspired by the multiscale testing procedure proposed in \citep{butucea2013detection}, we define our multiscale scan statistic as 
\beq \label{lik2}
\mscan(\bX) = \max_{I \subset [M], J \subset [N]}\bigg\{ \frac{\sum_{(i,j) \in  I \times J} X_{ij}}{\sqrt{|I||J|}} - \lambda_{|I|,|J|} \bigg\},
\eeq
where 
\beq \label{lambda}
\lambda_{m,n} = \sqrt{2\log\bigg[MN{M \choose m} {N \choose n}\bigg]}.
\eeq
Naturally, the corresponding estimator for localizing the anomalous submatrix will be a maximizer of \eqref{lik2}, namely
\beq \label{lik3}
\Phi_{\mscan}(\bX) = \argmax_{I \subset [M], J \subset [N]}\bigg\{ \frac{\sum_{(i,j) \in  I \times J} X_{ij}}{\sqrt{|I||J|}} - \lambda_{|I|,|J|} \bigg\}.
\eeq

\section{Two algorithms}
\label{sec:algo}

\subsection{The LAS algorithm}

The computation of \eqref{lik3} seems to require the scanning of all, or most, submatrices, and there are too many submatrices for this to be possible.  (The total number of submatrices is equal to $2^{M+N}$.) 
This fact makes the direct computation of the multiscale scan statistic computationally intractable. We also note that the computation of \eqref{lik3} is harder than that of \eqref{scan}, and the latter is proved to be NP-hard by \citep{cheng2000biclustering}.  It seems therefore necessary to resort to approximations.

When the true anomaly size $(m^*,n^*)$ is known, \citep{shabalin2009finding} proposed an iterative hill-climbing algorithm, called LAS, to approximate the statistic in \eqref{lik2}; see \algref{shabalin}. 
As a hill-climbing algorithm, it may be trapped in local maxima, so that in practice the algorithm is run on several (random) initializations \citep{butucea2013detection, arias2017distribution}. 

\begin{algorithm}
	\KwIn{submatrix size $(m^*, n^*)$, initial row index set $\hat{I}$ with $|\hat{I}| = m^*$}
	\KwOut{$\hat{I}$, $\hat{J}$}
	Compute the sum $S_j = \sum_{i \in \hat{I}} X_{ij}$, and let $\hat{J}$ be the column index set corresponding to the largest $n^*$ items in $\{S_j \}_{j =1}^N$; \\
	Compute the sum $T_i = \sum_{j \in \hat{J}} X_{ij}$, and let $\hat{I}$ be the row index set corresponding to the largest $m^*$ items in $\{T_i \}_{i=1}^M$; \\
	Repeat \textit{Steps 1 and 2} until convergence.
	\caption{Large Average Submatrix (LAS) \citep{shabalin2009finding}}
	\label{alg:shabalin}
\end{algorithm}

\subsection{The adaptive LAS algorithm}
We modify \algref{shabalin} in order to approximate the multiscale scan statistic.  In principle, we could run \algref{shabalin} on each submatrix size $(m,n) \in [M] \times [N]$, but although computationally feasible, it remains computationally demanding.  Instead, we adopt a more greedy approach.  We still rely on alternatively maximizing the objective, over the column index set, then the row index set, etc, until convergence, but the objective is that of \eqref{lik2}.  See \algref{adapt} for details.  In practice, the algorithm is run on multiple random initializations as well. 

\begin{algorithm}
	\KwIn{initial submatrix size $(m, n)$, initial row index set $\hat{I}$ with $|\hat{I}| = m$}
	\KwOut{$\hat{I}$, $\hat{J}$}
	Run \algref{shabalin} based on the initial input, obtaining $(\hat{I}, \hat{J})$\\
	Compute the sums $S_j = \sum_{i \in \hat{I}} X_{ij}$ and $S^\dagger_j = \frac1{\sqrt{j}} \sum_{k=1}^{j} S_{(k)} - \lambda_{|\hat{I}|, j}$ where $S_{(1)} \ge S_{(2)} \ge \cdots \ge S_{(N)}$ are the ordered $\{S_j \}_{j =1}^N$; \\
	 Let $n$ maximize $S^\dagger_j$ over $j = 1, \dots, N$, and let $\hat{J}$ be the column index set corresponding to the largest $n$ items in $\{S_j \}_{j =1}^N$; \\
	Compute the sums $T_i = \sum_{j \in \hat{J}} X_{ij}$ and  $T^\dagger_i = \frac1{\sqrt{i}} \sum_{l=1}^{i} T_{(l)} - \lambda_{i, |\hat{J}|}$ where $T_{(1)} \ge T_{(2)} \ge \cdots \ge T_{(M)}$ are the ordered $\{T_i\}_{i =1}^M$; \\
	 Let $m$ maximize $T^\dagger_i$ over $i = 1, \dots, M$, and let $\hat{I}$ be the row index set corresponding to the largest $m$ items in $\{T_i\}_{i =1}^M$; \\
	Repeat \textit{Steps 2-5} until convergence.
	\caption{Adaptive LAS}
	\label{alg:adapt}
\end{algorithm}

Note that the initial submatrix size input $(m, n)$ in \algref{adapt} does not necessarily represent any prior knowledge of the anomaly size, although it can be informed by such prior knowledge.

\subsection{A golden section search variant}
We propose a variant based on the (two-dimensional) golden section search, a well-known method for finding the maximum of a unimodal function (see \citep{chang2009n} for a general survey of the algorithm in $N$ dimensions). Define
\beq \label{target}
f_\bX(m,n) = \frac{\scan_{m,n} (\bX)}{\sqrt{mn}} - \lambda_{m,n},
\eeq
so that the multiscale scan statistic is the maximum value of $f_\bX$:
\beq\notag
\mscan(\bX) = \max_{m,n} f_\bX(m,n).
\eeq
We simply apply the two-dimensional golden section search (with some modifications for discreteness) to $f_\bX$.  Details are provided in \algref{gss}.  

The hope is that, if the signal strength is reasonably large and our initial bound on the submatrix size, $(\bar m, \bar n)$ below, is sufficiently accurate, $f_\bX$ is (with high probability) unimodal on $[\bar{m}] \times [\bar{n}]$ with maximizer $(m^*, n^*)$. 
We are not able to support this with a theoretical result, but have performed some numerical experiments, some of them reported in \secref{gss}, that seem to support this.

In the algorithm, instead of computing $f_\bX(m,n)$, we instead apply \algref{shabalin} to provide an approximation.  In the description, $\phi$ is the golden ratio constant, namely, $\phi = 0.5(\sqrt{5} - 1)$. 

\begin{algorithm}
	\KwIn{initial submatrix size $(\bar m, \bar n)$}
	\KwOut{$\hat{I}$, $\hat{J}$}	
	Denote $(m_{\text{min}}, n_{\text{min}} )= (1,1)$, $(m_{\text{max}},n_{\text{max}}) = (\bar m, \bar n)$; \\
	Compute $m_1 =\ceil*{m_{\text{max}} + (m_{\text{min}} - m_{\text{max}}) \phi}, m_2 =\floor*{ m_{\text{min}} + (m_{\text{max}} - m_{\text{min}}) \phi } $, $n_1 = \ceil*{ n_{\text{max}} + (n_{\text{min}} - n_{\text{max}}) \phi}, n_2 =\floor*{  m_{\text{min}} + (n_{\text{max}} - n_{\text{min}}) \phi } $;  \\
	Compute $f_\bX(m_1,n_1), f_\bX(m_2,n_1), f_\bX(m_1,n_2), f_\bX(m_2,n_2)$, denote $(m^\#, n^\#) = \arg \max_{(i,j) \in \{1,2\}^2} f_\bX(m_i,n_j)$ ;\\	
	Following the standard Golden Section Search algorithm for continuous functions, renew $(m_{\text{max}}, n_{\text{max}} )$, $(m_{\text{min}}, n_{\text{min}} )$ based on $(m^\#, n^\#)$;\\
	Repeat \textit{Steps 2-4}, till $\max(m_{\text{max}} - m_{\text{min}}, n_{\text{max}} - n_{\text{min}}) \leq 3$;\\
	Compute $f_\bX(s,t)$ for all $(s,t) \in [m_{\text{min}} , m_{\text{max}}] \times [ m_{\text{min}} ,m_{\text{max}}] \cap \mathbb{N}^2$. Denote the maximal among the results as $f_\bX(\hat{m}, \hat{n})$. Finalize $(\hat{I},\hat{J})$ as the output of \algref{shabalin} with input $(\hat{m}, \hat{n})$.
	\caption{Golden Section Search Algorithm}
	\label{alg:gss}
\end{algorithm}

We note that the algorithm needs to stop the loop when the search frame is $3 \times 3$ or smaller, for otherwise the loop would not break because of the discrete nature in the index sets (the search frame will not shrink). When this happens, the algorithm switches to an exhaustive search.
The initial submatrix size $(\bar{m},\bar{n})$ is ideally an upper bound on the actual submatrix size.  Its choice can be informed by prior knowledge.

\section{Theoretical property}
\label{sec:theory}

\subsection{Gaussian entries}

We establish a performance result for our multiscale scan statistic \eqref{lik3}.  Recall from \secref{mle} that the true anomaly $I^* \times J^*$ has size $(m^*,n^*)$. As before, we allow $\theta$ to change with $(M,N,m^*,n^*)$.
 
\begin{thm}[Exact recovery] \label{thm:theorem}
	Consider the model \eqref{normal} as described in \secref{introlocal} in an asymptotic regime where \eqref{assume} holds.
	Set  
\beq \label{theta1}
\theta_1 
= C\max \bigg\{ 
\sqrt{\frac{\log (M-m^*) + \log m^*}{n^*}} ,  \sqrt{ \frac{\log (N-n^*) + \log n^*}{m^*} } , \frac{\lambda_{m^*, n^*}}{\sqrt{m^*n^*}}, \bigg\},
\eeq
where $C$ is the positive real solution to the equation
\beq \label{constant}
C = 2\big( [C/(C-1)]^{1.5} + [C/(C-1)]^{1.25}\big).
\eeq
If   
\beq \label{cond}
\liminf \frac{\theta}{\theta_1  } > 1,
\eeq
the estimator defined by \eqref{lik3}
is equal to $(I^*, J^*)$ with high probability as $M,N,m^*,n^* \to \infty$.
\end{thm}

\begin{rem}
We note that $C \approx 4.32$. By comparing the terms defining $\theta_1$ with the terms defining $\theta_0$ (first defined in \thmref{bis}), it is easy to see that $\theta_1 / \theta_0 = O(1)$.
\end{rem}

\begin{rem}
The first two terms in the maximum \eqref{theta1} are associated with the row and column structures of the anomaly and data matrix, while the last part is associated with the existence of weak consistent estimators. In fact, \citep{hajek2017information, hajek2017submatrix} show that when $(m^*, n^*)$ are known, the scan statistic \eqref{scan} is weakly consistent if 
\beqn 
\liminf \theta \bigg\{ \frac{\lambda_{m^*, n^*}}{\sqrt{m^*n^*}}, \bigg\}^{-1} > 1.
\eeqn
(There is weak consistence if $\E (|(I^*  \times J^*) \symd \Phi_{\scan} | / m^*n^*) \to 0$, where $\symd$ denotes the set symmetric difference.) 
\end{rem}

\begin{proof}[Proof strategy]

Denote $S^* = I^* \times J^*$, which identifies the submatrix. 
Key to the proof is to bound the probability that, for some other submatrix $S$ of size $(m,n)$,
\beq \label{event}
\frac{\sum_{(i,j) \in  S} X_{ij}}{\sqrt{mn}} - \lambda_{m,n} > \frac{\sum_{(i,j) \in S^*} X_{ij}}{\sqrt{m^*,n^*}} - \lambda_{m^*,n^*}.
\eeq
The union of such events, as $S$ runs through all matrices, is exactly the event where the multiscale scan fails at exact recovery.

We will start by defining some quantities which quantify the closeness between the candidate matrix $S$ and true anomaly $S^*$.
For the submatrices in $\cS_{m,n}$, which is the collection of all submatrices with dimensions $(m,n)$, define their affinity in size as
\beq \notag
\rho_{m,n} = \max_{S \in\, \cS_{m,n}}  \frac{|S \cap S^*|}{\sqrt{|S| |S^*|}} = \frac{(m \wedge m^*)}{\sqrt{m m^*}} \frac{(n \wedge n^*)}{\sqrt{n n^*}}.
\eeq

For a specific submatrix $S$ with size $(m,n)$, such that $S \cap S^*$ is a $s\times t$ submatrix, define the affinity of coverage as
\beq \notag
\upsilon_{m,n,s,t} = \frac{st}{\sqrt{mm^*nn^*}}.
\eeq

The proof strategy is that, with probability converging to $1$, we can sequentially rule out the submatrices with low affinity in size and low affinity in coverage.  We then analyze those submatrices left behind, which have size and coverage comparable with the true anomaly. 

In fact, the introduction of the regularizer $\lambda_{m,n}$ in the expression of the multiscale scan statistic is there to eliminate the submatrices with low affinity in size or in coverage. While when the affinities are not to small (bounded from zero by a positive number), the number of candidate matrices is controlled. We measure the probability of \eqref{event} when $S$ and $S^*$ only differ by one column or one row, and show that this is the largest failure probability among all candidates that have affinities bounded from below. 
\end{proof}

\subsection{An extension to an exponential family}
\label{sec:exponential}

In practice the Gaussian assumption is not satisfied in many situations. For example the entries may only take integer values (as with counting data) or binary data (as with presence-absence matrices in Ecology \citep{gotelli2000null}). Therefore it is of importance to extend the result to a distribution family that covers several common data types. Following \citep{butucea2013detection,arias2018distribution,arias2017distribution}, we consider a one parameter exponential family.  Let $\nu$ denote a distribution on the real line with mean zero and variance $1$. Denote $\varphi(\theta)$ as the moment generating function of $\nu$, meaning $\varphi(\theta) = \int e^{\theta t } \nu (dt) $, and suppose that $\varphi(\theta) < \infty$ for all $\theta$ in $[0, \theta_\star)$. Here $\theta_\star$ is defined as $\sup \{ \theta: \varphi(\theta) < \infty\}$ and could be equal to infinity. For $\theta \in [0, \theta_\star)$, define 
\beq \label{family}
f_\theta (x)= \exp \{ \theta x  - \log \varphi(\theta) \},
\eeq
which is a density with respect to $\nu$.

By selecting $\nu$ appropriately, the distribution family becomes the normal location family, the Poisson family, or the Rademacher family. Note that with $\nu$ fixed, the distribution family is stochastically monotone in $\theta$ (\cite{MR2135927}, Lemma 3.4.2), which is to say, for $X_1 \sim f_{\theta_1}$ and $X_2 \sim f_{\theta_2}$ with $0\leq \theta_2 \leq \theta_1 \leq \theta_\star$ and fixed $\nu$,
\beq \label{monotone}
\P (X_1 \geq t) \geq \P(X_2 \geq t), \forall t \in \mathbb{R}.
\eeq 
This fact enables us to model the submatrix localization problem with $\theta$ controlling the signal strength. With $\theta$ increasing, the anomaly is more 'anomalous', making the localization problem relatively easier to solve. The localization problem can now be formalized with $\nu$ as the role of noise, and entries in the anomaly are distributed as $f_\theta$. Formally, 
\beq \label{exponential}
X_{ij} \sim \left\{
                \begin{array}{ll}
                  f_\theta \quad  (i,j) \in I^* \times J^*\\
                  \nu \quad  (i,j) \notin I^* \times J^*
                \end{array}
              \right.
\eeq 

\begin{thm}[Lower bound under exponential family assumption] \label{thm:exponential1}
Under the assumption of \eqref{family} and suppose \eqref{assume} holds. 
Additionally, we assume 
\beq \label{addass1}
\frac{\max (\log M, \log N)}{\min (m^* , n^*)} \to 0.
\eeq 
Then if 
\beq\notag
\theta \le C \max \bigg( \sqrt{\frac{\log M}{n^*}}, \sqrt{\frac{\log N}{m^*}} \bigg),
\eeq
with $C < 1/\sqrt{8}$, even with full knowledge of $(m^*, n^*)$, there does not exist a strongly consistent estimator of $(I^*, J^*)$.
\end{thm}

\begin{proof}[Proof Sketch of \thmref{exponential1}]
The proof is following most of the arguments in the proof of Theorem 1 in \citep{kolar2011minimax}.  The result focuses on the normal location family, but extends to other one-parameter exponential families. We aim to prove the following:
\beq\label{KL}
\mathcal{D} (\mathcal{P}_0 ||\mathcal{P}_j  ) \geq m^*\theta^2(1 + o(1)),
\eeq
where $\mathcal{D} (\cdot || \cdot) $ is the Kullback Leibler divergence between two distribution, and $\mathcal{P}_0, \mathcal{P}_j$ are the distributions of data $\bX$ when $S^* = [m^*] \times [n^*]$ and $S^* = [m^*] \times \{[n^*-1] \cup \{j\}\}$, with $j \in \mathbb{N} \cap [n^*+1,N]$, respectively. Replace (15) in the proof of Theorem 1 in \cite{kolar2011minimax} with the above inequality, and the result follows.
\end{proof}

This result also partially answers the open problem raised by \citep{butucea2015sharp} asking for a lower bound under a general exponential family.

Before we head into evaluating the performance of our proposed estimator \eqref{lik3} applied under the exponential family assumption, we need to re-define the adjusting quantity in \eqref{lambda}. 
Fixing a constant $\delta > 0$, we re-define $\lambda_{m,n}$ as 
\beq \label{lambda2}
\lambda_{m,n} = \sqrt{(2+ \delta)\log\bigg[MN{M \choose m} {N \choose n}\bigg] }.
\eeq
The multiscale scan statistic, and its associated estimator of the anomalous submatrix, is defined again according to \eqref{lik2} and \eqref{lik3}.

\begin{thm}[Exact recovery under exponential family] \label{thm:exponential2}
Assume the model \eqref{family}, and suppose that \eqref{assume} and \eqref{addass1} hold.  Take $\underline{m}$ and $\underline{n}$ such that $\log(N) = o(\underline{m})$ and $\log(M) = o(\underline{n})$.
Suppose we scan over submatrices of size $(m,n)$ such that $m \ge \underline{m}$ and $n \ge \underline{n}$, and that the anomalous submatrix is among these, meaning that $m^* \ge \underline{m}$ and $n^* \ge \underline{n}$.
Suppose $\theta_1$ is as in \thmref{theorem}.
If 
\beq\notag
\liminf \frac{\theta}{  \theta_1 } > 1,
\eeq
with high probability as $M,N,m,n \to \infty$, this scan returns the anomalous submatrix.
\end{thm}

\begin{rem}
There is a normal approximation underneath which drives the asymptotic behavior of the statistic to be close to that under the normal model.  This explains why the result is similar under a general exponential family as it is under the normal family. In fact, the proof of \thmref{exponential2} is otherwise essentially the same as that of \thmref{theorem}.
\end{rem}

\section{Numerical experiments}
\label{sec:numerics}

\subsection{Performance of Proposed Algorithms}
We performed some simulation experiments to evaluate the performance of our proposed algorithms.  We generate a matrix with independent entries according to \eqref{normal} as well as \eqref{exponential}, and perform the proposed algorithms (Adaptive LAS and Golden Section Search LAS) on the generated data. 
By permutation invariance, we simply choose $I^* = [m^*]$ and $J^* = [n^*]$.
We then measure the accuracy of an estimator $(\hat{I}, \hat{J})$ as follows
\beq \label{errmea}
\text{Err} ( \hat{I}, \hat{J}) = \log (|\hat{I} \symd [m^*]| + |\hat{J} \symd [n^*]| + 1 ).
\eeq

We compare the proposed algorithms with two computationally tractable algorithms with proved consistency. 
\begin{itemize}
\item Spectral algorithm of \citep{cai2017computational}. The algorithm computes the singular value decomposition of the data matrix $\bX$, then applies $k$-means algorithm on the first left and right singular vectors with number of clusters set to $k = 2$, and clusters the row and column indices according to corresponding singular vector's clustering results.
\item Greatest Marginal Gap method of \citep{brault2016fast}. The algorithm calculates row and column sums of the data matrix $\bX$, denoted as $\{RS_i\}$ and $\{CS_i\}$, then sorts the row and column sums, obtaining $\{RS_{(i)}\}$ and $\{CS_{(i)}\}$. The algorithm clusters rows based on $i_R = \arg\max_i RS_{(i+1)} - RS_{(i)}$, yielding two sets $\{RS_{(1)}, \ldots, RS_{(i_R)}\}$ and $\{RS_{(i_R + 1)}, \ldots, RS_{(M)}\}$. The clustering of columns is analogous.
\end{itemize}

We consider a balanced setting where $(M,N,m^*,n^*)$ is taken to be $(1000,1200,170,140)$, and also an unbalanced setting where $(M,N,m^*,n^*)$ is taken to be $(4000,500,70,250)$. 
For each fixed signal strength, the data is independently generated $30$ times and the error counts defined by \eqref{errmea} for all four the algorithms are recorded. Three one-parameter exponential models (normal, Poisson, Rademacher) are investigated to illustrate the performance of the algorithms under different distributions.
See \figref{plot1} and \figref{plot2}. 

\subsubsection{Signal strength}
The signal strength is quantified by the parameter $\theta$ of the entries inside the submatrix. Denote the following quantity $\theta_{\text{crit}}$:
\beq\notag
\theta_{\text{crit}} = \max \bigg( \sqrt{\frac{\log M}{n^*}} ,\sqrt{\frac{\log N}{m^*}}, \sqrt{\frac{\log M + \log N}{m^* + n^*}}\bigg).
\eeq
This is part of the quantity on the left hand-side of \eqref{cond}, inside \thmref{theorem}. 
We zoom in to the interval $[1.0 \times \theta_{\text{crit}} , 4.0 \times \theta_{\text{crit}}]$. The signal strength is increased in the process of simulation, each time by $0.1 \times \theta_{\text{crit}}$.

\subsubsection{Simulation result}

The two proposed algorithms perform similarly under the balanced design, across all three data types. See \figref{plot1}. Adaptive Hill-Climbing has a slightly weaker performance when the signal is weak (around $1\times \theta_{\text{crit}}$ to $1.5\times \theta_{\text{crit}}$), but the two algorithms' error rates shrink to zero at the same signal strength, showing that the two approximate algorithms return the same result when signal strength goes beyond some threshold, and the result would exactly recover the planted anomaly.

As \figref{plot2} shows, the two algorithms perform differently under the imbalanced design when the signal is weak, with the Golden Section Search performing better. This is due to the fact that the search space of GSS is smaller, as well as the existence of local maximums when the design is imbalanced (see \figref{gss} for an example). However, when considering the successful rate of exact recovery, the signal strength, above which the algorithm's error rate shrinks to zero is still similar for both algorithms.

It is worth mentioning that in both cases, our proposed methods
outperform the other two computationally tractable methods. In
the balanced case, both Adaptive Hill-Climbing and Golden Section Search beat the spectral method by a small margin in error rate, while the error rate of Greatest Marginal Gap method is much worse in the examined signal strength region. In the imbalanced case, Adaptive Hill-Climbing has similar performance with spectral method in the weak signal region, while the GSS again outperforms both, showing that the multiscale scan statistic has better discovery power compared to spectral method. Again, the error rate of Greatest Marginal Gap is much larger.

\subsection{Computing time}

Here we present the computing time for executing the programs described in the previous section. The simulated data is from model \eqref{normal}, with corresponding sizes $(M,N,m^*, n^*) = (1000,1000,100,100)$ and signal strength $\theta = 2.5\times \theta_{\text{crit}}$. We generate data under this model $100$ times, and each time record the computation time for each of the competing methods.

\clearpage
\begin{figure*}%
	\includegraphics[width=0.9\textwidth]{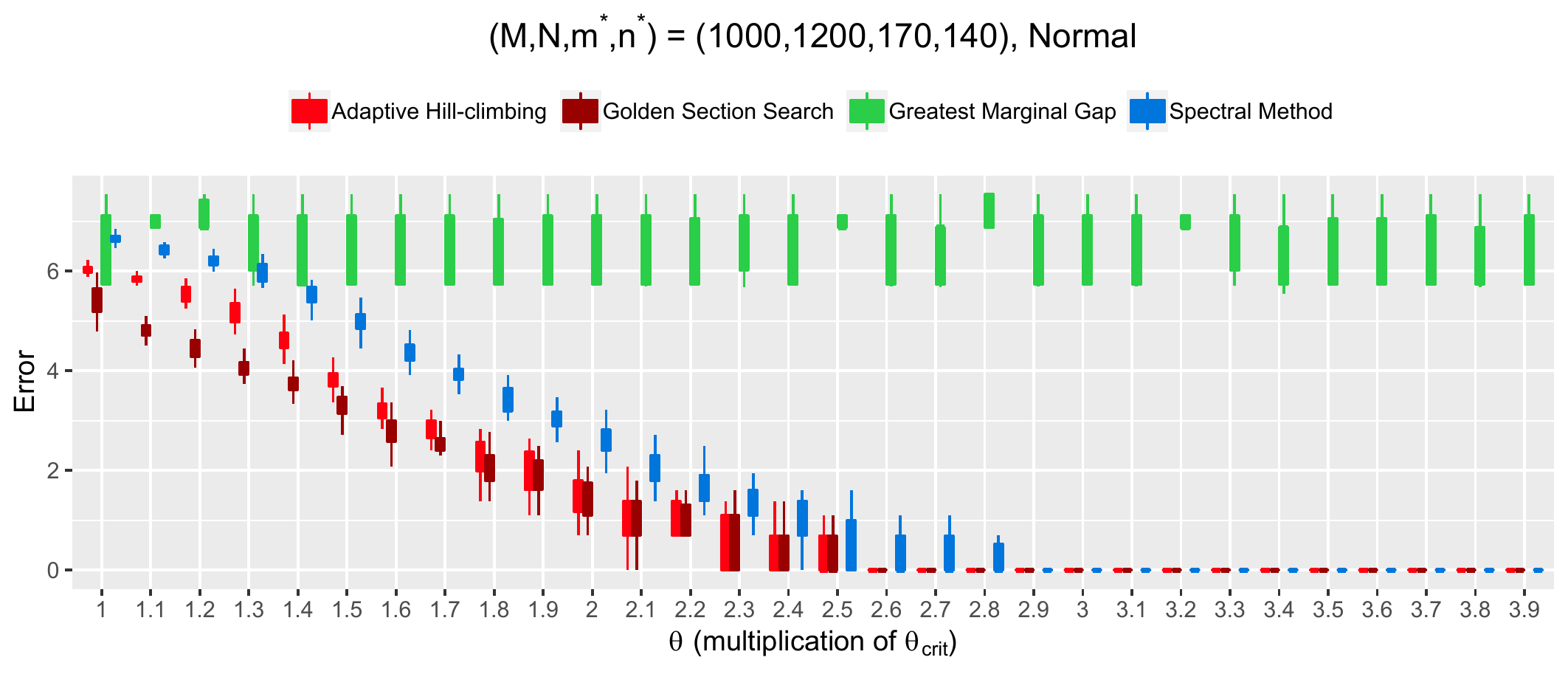} \\%

	\includegraphics[width=0.9\textwidth]{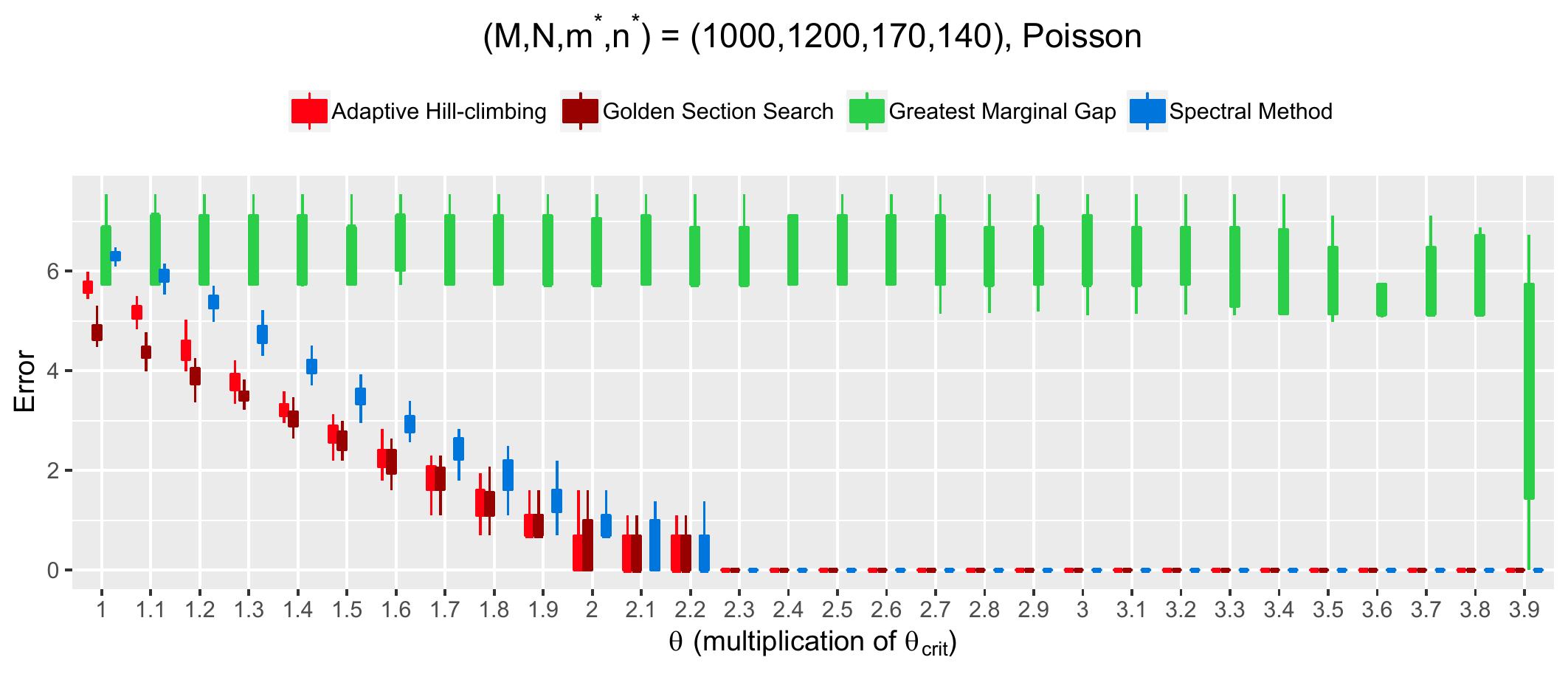} \\%

	\includegraphics[width=0.9\textwidth]{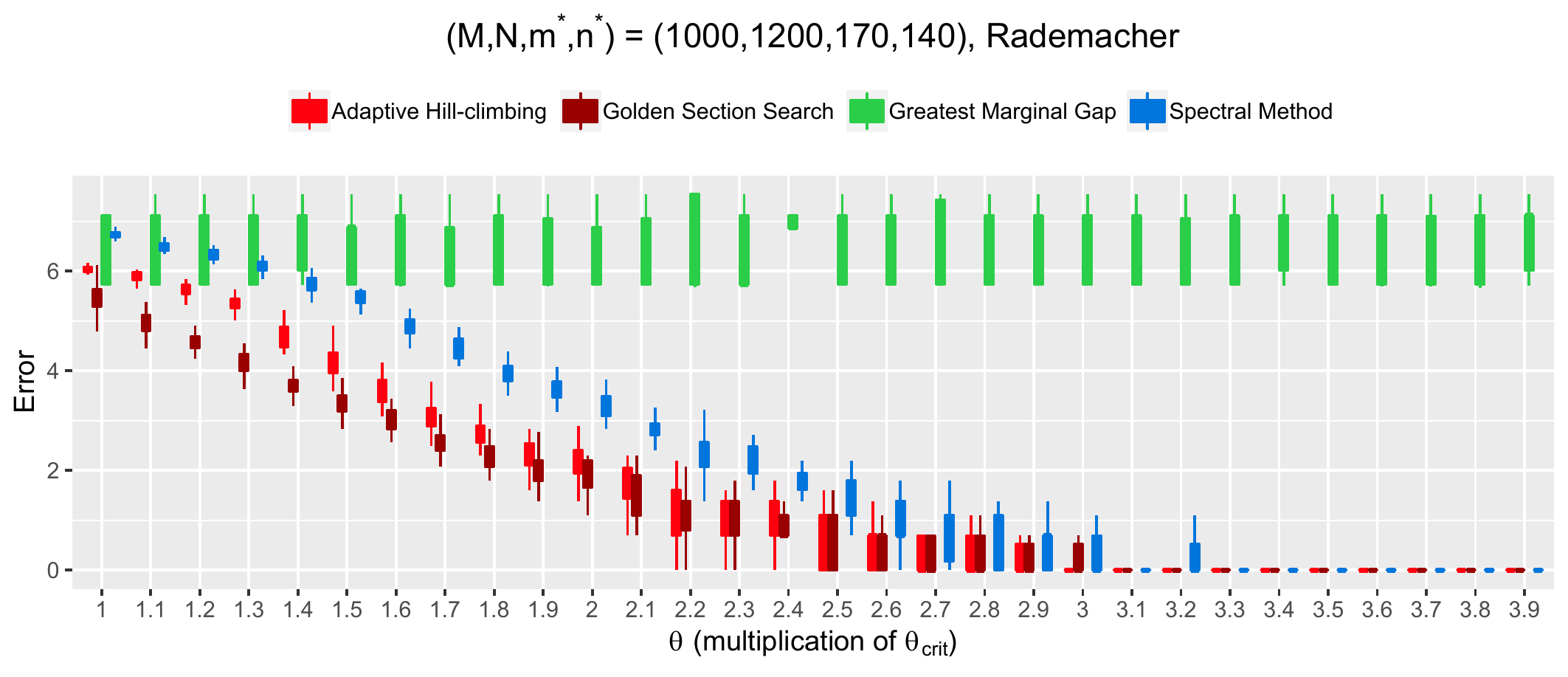} \\%
	\caption{Error counts for a balanced design}%
	\label{fig:plot1}%
\end{figure*}

\clearpage

\begin{figure*}
	
	\includegraphics[width=0.9\textwidth]{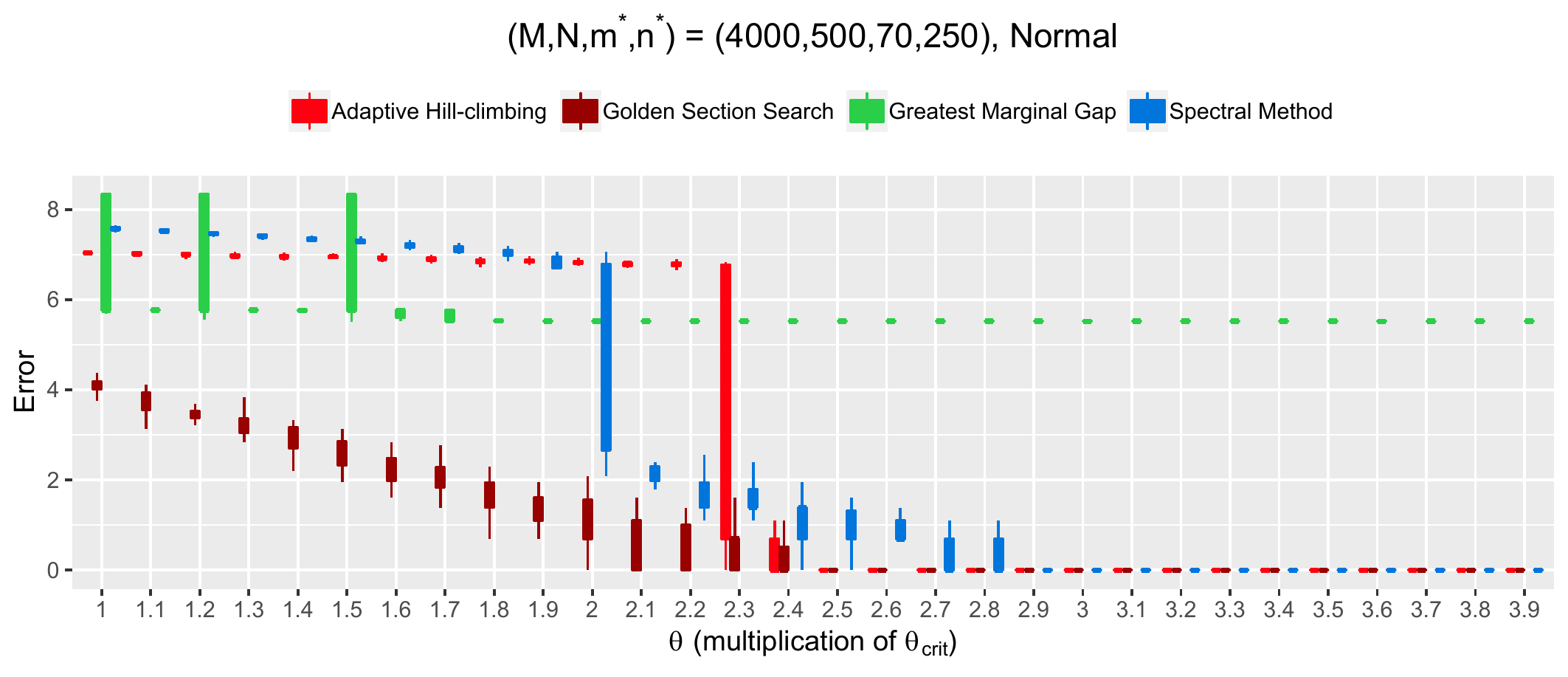} \\%
	 
	\includegraphics[width=0.9\textwidth]{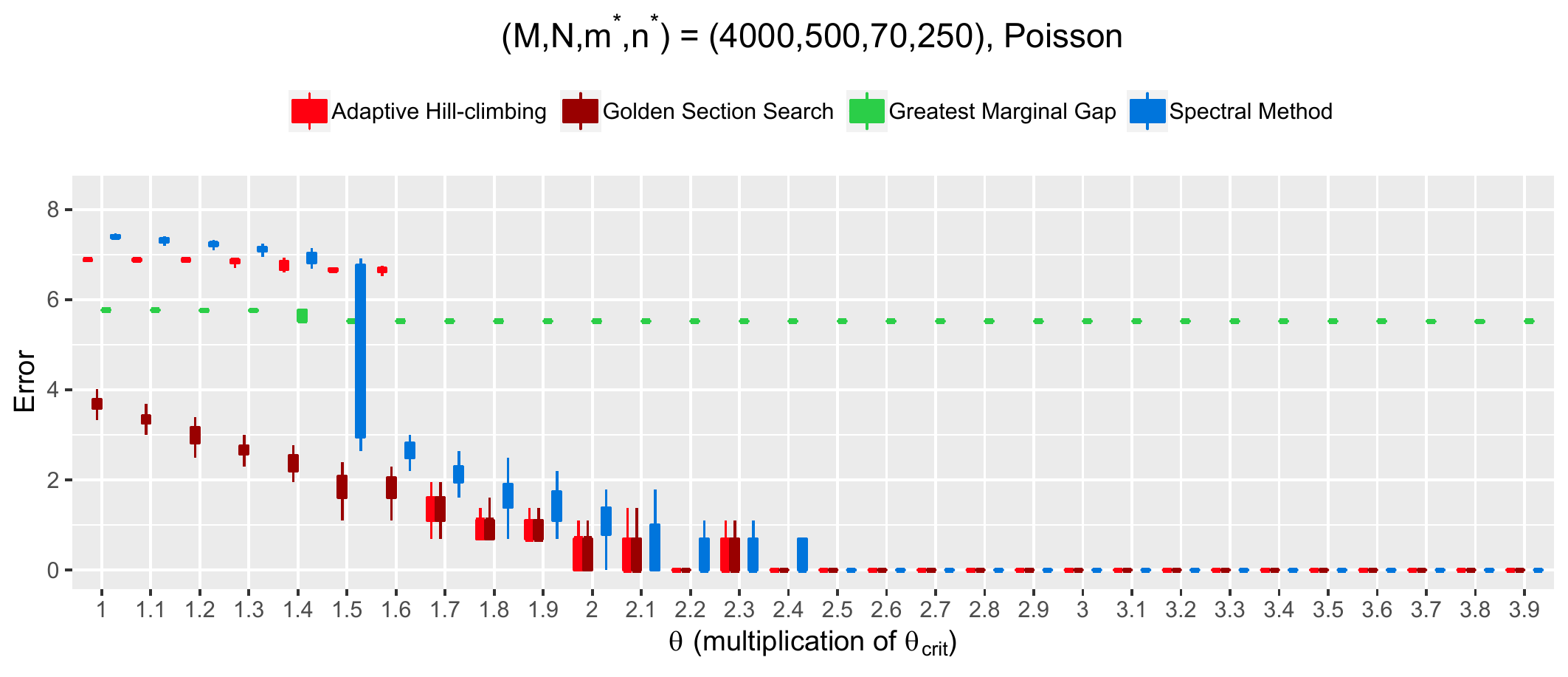} \\%
	 
	\includegraphics[width=0.9\textwidth]{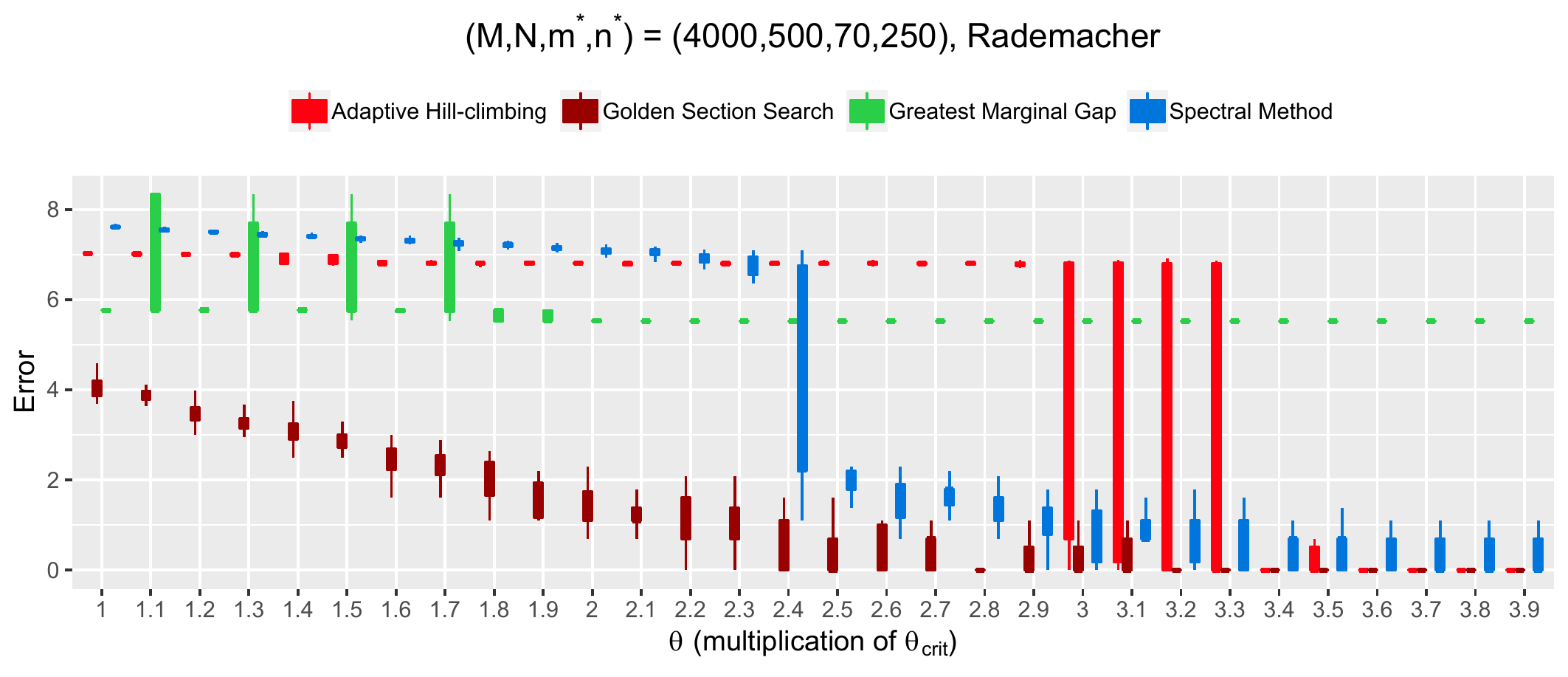} \\%
	\caption{Error counts for an imbalanced design}%
	\label{fig:plot2}%
\end{figure*}
\clearpage

\begin{figure}[H]%
	\centering
	\subfloat{{\includegraphics[width=.5\textwidth]{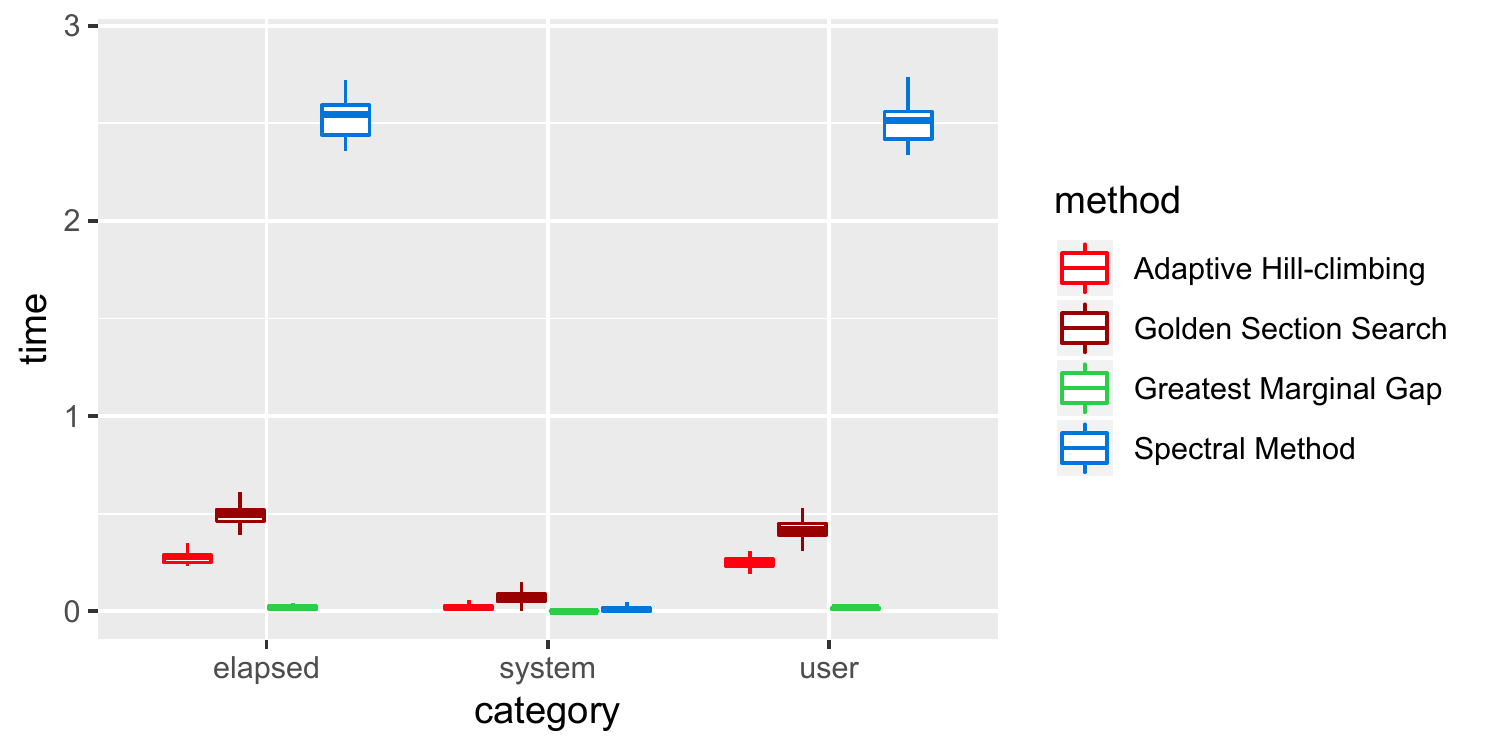} }}\\%
	\caption{Computing times for different algorithms. (The computing time is calculated by function \textsf{`proc.time'} in programming language \textsf{R}. Category `user' represents the time for executing the program codes, `system' is the CPU time charged for execution by the system on behalf of the calling process, and `elapsed' is the sum of the other two.)}%
	\label{fig:time}%
\end{figure}

\begin{figure}[H]%
	\centering
	\subfloat{{\includegraphics[width=0.45\textwidth]{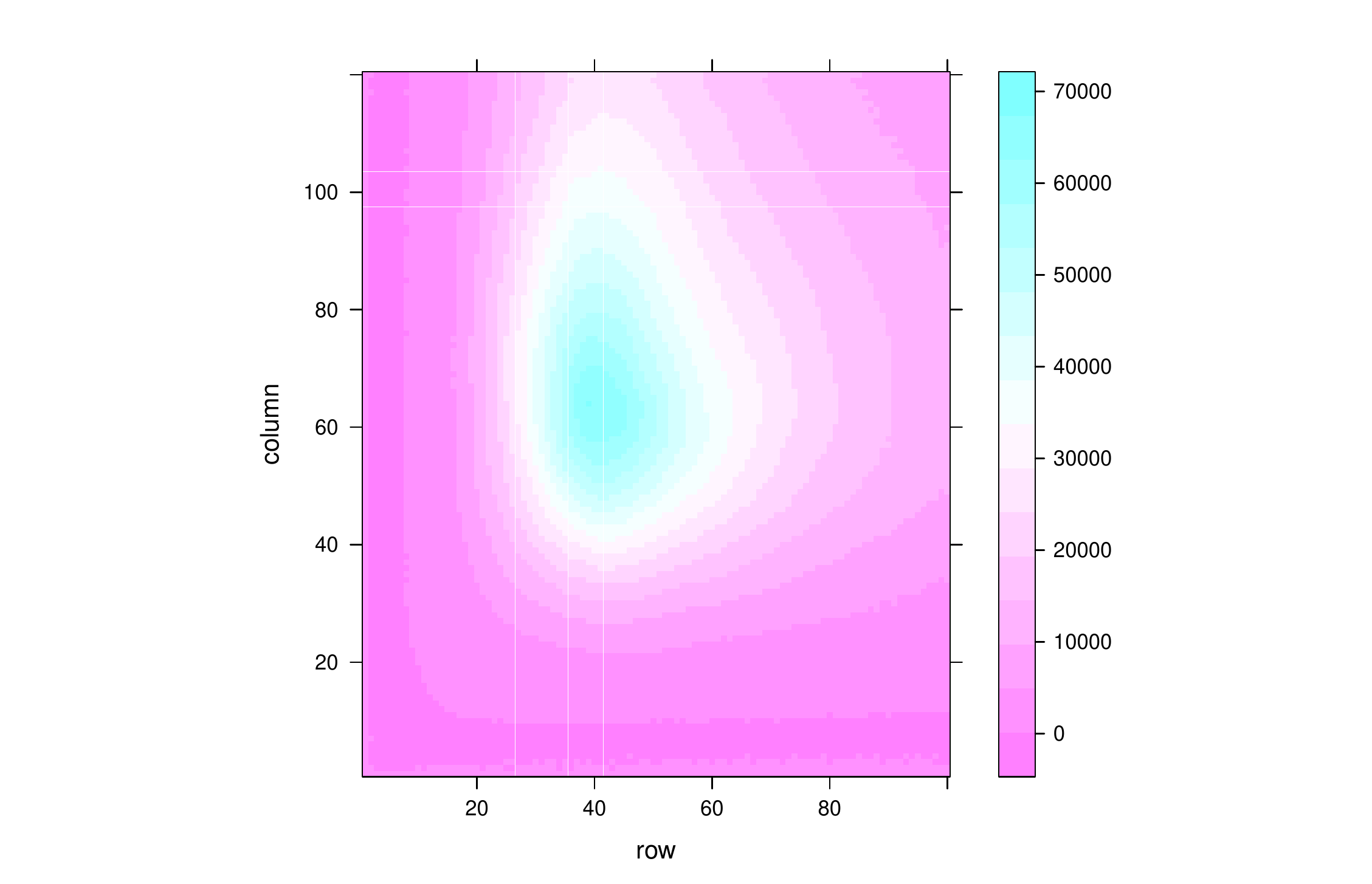} }}\\%
	\subfloat{{\includegraphics[width=0.45\textwidth]{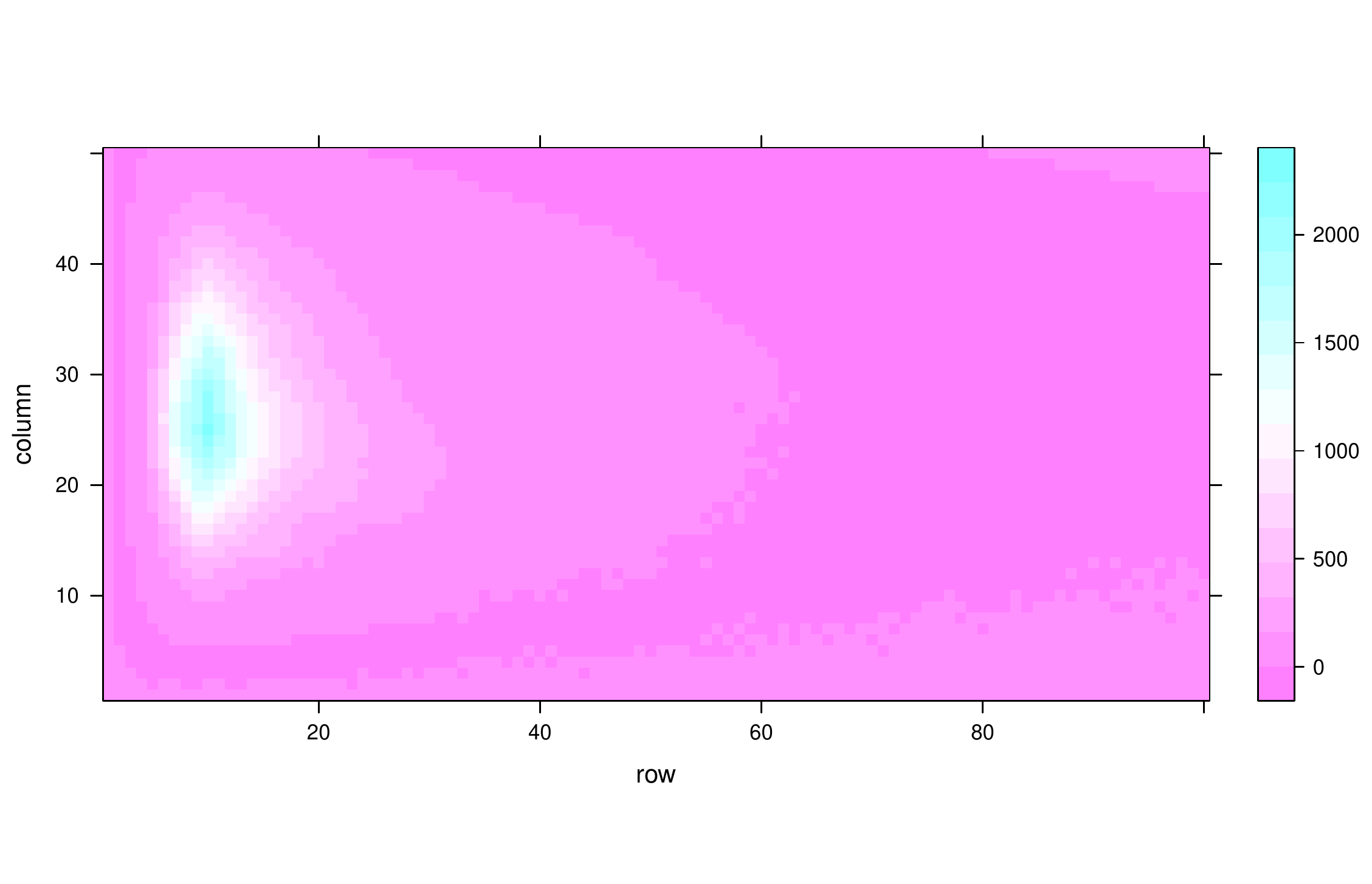} }}\\%
	\caption{Levelplots for illustrating unimodality}%
	\label{fig:gss}%
\end{figure}

The spectral method was computationally much more demanding than the other methods. As expected, the clustering algorithm based on the largest marginal gap is the fastest, however our algorithms are not too far (while much more accurate).

\subsection{Unimodality issue in \algref{gss}} \label{sec:gss}

The golden section search presented in \algref{gss} requires the function defined in \eqref{target} to be unimodal, in order to successfully discover the global maximal. In general, checking the unimodality of a function is often hard and here we present some numerical experiment illustrating the unimodality of the target function \eqref{target}.

We use two simulated Gaussian datasets. For each combination of $(M,N,m^*,n^*)$, we set the level of signal to $\theta = 2\times \theta_{\text{crit}}$, which is just above the signal level such that the search  algorithm makes few to zero mistakes. Then for every pair of $(m,n)$ such that $m \leq \bar{m}, n\leq \bar{n}$, we calculate the function value according to \eqref{target}. The scan statistic $\scan_{m,n}(\bX)$ is calculated by \algref{shabalin}, and $\lambda_{m,n}$ is approximated by 
\smash{$(2(\log M + \log N + m\log (M/m) + n\log (N/n))^{1/2}$}.
Here we use the fact that $\log {K \choose k} \approx k\log (K/k)$ when $k \ll K$. The simulation is otherwise set as in the previous section, in that we examine both a balanced and an imbalanced setting. In the balanced setup, $(M,N,m,n) = (300,360,40,60) $ and $(\bar{m}, \bar{n}) = (100,120)$. In the imbalanced setup, $(M,N,m,n) = (500,50,10,25) $ and $(\bar{m}, \bar{n}) = (100,50)$.

To better illustrate the unimodality, the simulated values of $f$ is raised to its fourth moment to enlarge the difference around the mode.  We can see from \figref{gss} that except for a few points around the edge, the target function \eqref{target} has a clear unimodal structure on the majority of the search field $[\bar{m}] \times [\bar{n}]$.

\section{Conclusion and discussion}

In this paper we propose a new multiscale scan statistic for localizing the anomalous submatrix inside a large noisy matrix, which does not require prior knowledge on the anomalous submatrix size. We show the signal strength needed for strong consistency, and design two algorithms with good approximating accuracy and computing speed. There are, however, some problems for future work and discussions.

\textbf{Minimaxity of the multiscale scan statistic}: \citep{butucea2015sharp} showed a sharp minimax signal bound, as we described in \thmref{bis}. Can our estimator based on the multiscale scan statistic reach the same minimax bound (which relies on knowledge of the submatrix size)? While we are unable to reduce the constant in \thmref{theorem} to the bound, we conjecture that our estimator is essentially minimax. Another measurement of accuracy of estimators is weak consistency. \citep{hajek2017submatrix, hajek2017information} shows that \eqref{scan} is minimax in the weak consistency sense. When is the multiscale scan statistic weak consistent and is it minimax?

\textbf{Proof of unimodality of \eqref{target}}: What is the relationship of $(M,N,$ $m^*,n^*,\bar{m},\bar{n})$ and $\theta$, such that \eqref{target} is unimodal on $[\bar{m}]\times [\bar{n}]$ with high probability? Solution to this problem will directly lead to the success guarantee of \algref{gss}.

\textbf{Computationally tractable algorithms}: There is a growing literature about using computationally tractable methods such as semidefinite programming to approximate NP-hard problems. (for example, \citep{chen2016statistical} on using SDP on finding the maximal likelihood estimator). 
Can such an algorithm be designed to approximate with guarantied accuracy the multiscale scan?

\textbf{Tight minimax bound for exponential family case}: As \citep{butucea2015sharp} mentioned, minimax theory on the case of exponential family is still open. We give a bound based on \citep{kolar2011minimax}, but this bound is not tight. Also, although the scan statistic under the exponential family setup is proved to be minimax in testing by \citep{butucea2013detection}, the analysis of its localization performance is still open.

%
\section{Acknowledgment}
The authors thank the reviewers for the helpful feedback, and Jiaqi Guo for help in simulation studies.

%
\bibliographystyle{chicago}
\bibliography{ref}

%

\clearpage

\appendix

\section{Simulation details for reproducibility}

In the appendix we list some detail setups of our simulation in order to provide essential information for reproducing the simulation results in the paper. All the simulation are done with statistical software \textsf{R} version 3.5.2. The graph is generated by graphing package \textsf{ggplot2}. The random number generator are set with seed $100$ by the \textsf{set.seed} function for each single graph. For the further reproduction purposes, all the simulation codes are available at \url{https://github.com/nozoeli/adaptiveBiclustering}.

\subsection{Simulation Setup for \figref{plot1} and \figref{plot2}}

We list setups used during our simulation, mainly regarding the inputs of \algref{adapt} and \algref{gss}.

For \algref{adapt}, the initial input $(m,n) = (25,25)$. The initial row index set $\hat{I}$ is uniformly randomly chosen from $[M]$. The algorithm is repeated $50$ times with I.I.D. chosen $\hat{I}$, and the result producing the largest multiscale scan statistic is returned.

For \algref{gss}, the input $(\bar{m}, \bar{n}) = (500, 500)$. When using \algref{shabalin} to calculate $f$, the initial row index is chosen with rows with the largest row sums across all columns (e.g. input $m = 20$, the initial row index set is made of the indices with top $20$ largest row sums), and \algref{shabalin} is run once to calculate the scan statistic inside $f$.

These setups are used across the calculation of all figures in \figref{plot1} and \figref{plot2}.

\subsection{Simulation Setup for \figref{time}}

The setup of \algref{adapt} during this part of simulation is $(m,n) = (5,5)$, with the rest setup the same as the previous section.

The setup of \algref{gss} is the same as the previous section.

The function used in spectral method is as follows. The singular vector decomposition is using the \textsf{R} function \textsf{svd}, and the $k$-means algorithm is calling \textsf{R} function \textsf{kmeans} with default setup and number of clusters $2$. 

\subsection{Simulation Setup for \figref{gss}}

The calculation of scan statistic during this part of simulation is repeating \algref{shabalin} $100$ times with I.I.D. initial row indexes uniformly randomly chosen from $[M]$, and return the result with the largest entry sum.

\end{document}